\documentclass[12pt,a4paper,reqno]{amsart}
\usepackage{amssymb}
\usepackage{latexsym}
\usepackage{exscale}
\usepackage{graphicx}

\calclayout

\allowdisplaybreaks

\numberwithin{equation}{section}
\newtheorem{theor}{Theorem}[section]
\newtheorem{lemma}[theor]{Lemma}
\newtheorem{defi}[theor]{Definition}
\newtheorem{corol}[theor]{Corollary}
\newtheorem{remark}[theor]{Remark}

\newcommand{\R}{\mathbb{R}}
\newcommand{\re}{\mathbb{R}}

\renewcommand{\emptyset}{\mbox{\rm \O}}

\newcommand{\bigchi}{\mathop{\mathchoice%
{\mbox{\Large$\chi$}}{\mbox{\large$\chi$}}{\mbox{\normalsize$\chi$}}%
{\mbox{\small$\chi$}}}\nolimits}

\def\essinf{\mathop{\rm ess\ inf\ }}

\begin{document}

\title[Weighted inequalities related to a Muckenhoupt--Wheeden problem]
{Weighted inequalities related to a  Muckenhoupt and Wheeden problem for one-side singular integrals}

\author[M.~S.~Riveros]{Mar\'{\i}a Silvina Riveros}
\address{Mar\'{\i}a Silvina Riveros \\ Facultad de Matem\'atica Astronom\'ia y F\'isica \\ Universidad Nacional de C\'{o}rdoba \\
CIEM (CONICET) \\ 5000 C\'{o}rdoba, Argentina}
\email{sriveros@famaf.unc.edu.ar}

\author[R.~E.~Vidal]{Ra\'ul Emilio Vidal}
\address{Ra\'ul Emilio Vidal \\  Facultad de Matem\'atica Astronom\'ia y F\'isica
\\ Universidad Nacional de C\'ordoba\\
CIEM (CONICET) \\ 5000 C\'{o}rdoba, Argentina}
\email{rauloemilio@gmail.com}

\thanks{ supported by CONICET, and SECYT-UNC}




\begin{abstract}
In this paper we obtain for $T^+$, a one-sided singular integral  given by a Calder\'on-Zygmund kernel with support in $(-\infty,0)$,   a $L^p(w)$ bound   when $w\in A_1^+$.  A. K. Lerner, S. Ombrosi, and C. P\'{e}rez proved in [ ``$A_{1}$ Bounds for Calder\'{o}n-Zygmund operators related to a problem of Muckenhoupt and Wheeden", Math. Res. Lett. \textbf{16} (2009), no. 1, 149-156]   that this bound is sharp with respect to $||w||_{A_1} $ and $p$ . We also give a  $L^{1,\infty}(w)$ estimate, for a  related  problem of Muckenhoupt and Wheeden for  $w\in A_1^+$  . We improve the classical results, for one-sided singular integrals, by putting in the inequalities a wider class of weights.
\end{abstract}

\maketitle

\vskip-1.2cm\null


\section{Introduction}

Let $M$ be the classical Hardy-Littlewood maximal operator and $w$  a weight (i.e.  $w\in L_{loc}^1(\re^n)$ and $w>0$).  C. Fefferman and
E.M. Stein  in \cite{FS} proved an extension of the classical weak-type (1, 1) estimate:
\begin{equation}\label{debilmaximal}
||Mf||_{L^{1,\infty}(w)}\leq c_n\int_{\re^n}|f(x)|Mw(x)\,dx,
\end{equation}
where $c_n$ depends only  on the dimension. This is a sort of duality for $M$. 
A consequence of this result, using an interpolation argument, is the following:  if $1 < p < \infty$ and  $p' = \frac{p}{p-1}$ then,
\begin{align*}
\int_{\re^n}(Mf(x))^pw(x)\,dx\leq c_np'\int_{\re^n}|f(x)|^pMw(x)\,dx.
\end{align*}

 B. Muckenhoupt and R. Wheeden in the 70's  conjectured that the analogue of (\ref{debilmaximal})
should hold for $T$, a singular integral operator, namely\\
\begin{equation}\label{conjMW}
\sup_{\lambda>0}\lambda w(\{x\in\re^n: |Tf(x)|>\lambda\})\leq C\int_{\re^n}|f(x)|Mw(x)\,dx.
\end{equation}
The best result along this line can be found in \cite {P} where $M$ is replaced by the slightly larger operator $M_{L(logL)^\epsilon}$, $ \epsilon> 0$,
\begin{align*}
||Tf||_{L^{1,\infty}(w)}\leq C\,2^{\frac{1}{\epsilon}}\int_{\re^n}|f(x)|M_{L(\log L)^\epsilon}w(x)\,dx.
\end{align*}
The one-sided version of  this result it was obtained in \cite{MCS}.

 M.  C. Reguera in \cite{CR}  and  M. C. Reguera   and  C.  Thiele in \cite{CRT} proved that the Muckenhout-Wheeden conjecture  is false. In \cite{CR}    the author give a first approach by putting  in the right hand side the dyadic  maximal operator.  In \cite{CRT} they disproved (\ref{conjMW}) for $T$ the Hilbert transform.

On the other hand there is a variant of the conjecture (\ref{conjMW}) which has a lot of interest, namely the weak Muckenhoupt-Wheeden conjecture.
The idea is to assume an a priori condition on the weight $w$. This condition can be
read essentially from inequality (\ref{conjMW}): the weight $w\in A_1$  if there is a finite
constant $C$ such that $Mw(x)\leq Cw(x) \quad a.e. x\in \mathbb{R}^n$. Denote $||w||_{A_1}$ the smallest of these $C$. The conjecture is the following:

Let $w \in A_1$, then
\begin{align*}
\sup_{\lambda>0}\lambda w(\{x\in\re^n: |Tf(x)|>\lambda\})\leq C ||w||_{A_1}\int_{\re^n}|f(x)|w(x)\,dx.
\end{align*}
In \cite{ASC} the authors exhibit a logarithmic growth
 \begin{equation}\label{debilTA1}
||Tf||_{L^{1,\infty}(w)}\leq C ||w||_{A_1} \log(e+||w||_{A_1}) ||f||_{L^1(w)},
\end{equation}
where $C$ only depends on $T$ and the dimension. It is believed that the results obtained by Lerner, Ombrosi, and P\'{e}rez are the best possible.
Recently F. Nazarov, A. Reznikov, V. Vasyunin, A. Volberg, proved that the weak Muckenhoupt-Wheeden conjecture is also false. See \cite{NRVV}.

 To prove this logarithmic growth result,  they  have
to study first the corresponding weighted $L^p(w)$ estimate for $1 < p < \infty$ and $w \in A_1$
 \begin{equation}\label{fuerteTA1}
||Tf||_{L^{p}(w)}\leq C pp' ||w||_{A_1}  ||f||_{L^p(w)},
\end{equation}
where $C$ only depends on $T$ and the dimension, being the result, this time, fully sharp. See \cite{ASC}.
As a consequence of (\ref{debilTA1}) and applying  the Rubio de Francia's algorithm  they also get the following result,
let $1 < p < \infty$, $w \in A_p$ and let  $T$ be a Calder\'on-Zygmund operator
 then
\begin{equation}\label{fuerteTAp}
||T||_{L^{p,\infty}(w)} \leq  C||w||_{A_p}\log(e+||w||_{A_p})||f||_{L^p(w)},
\end{equation}
where $C = C(n, p, T).$

The  first result of this kind obtaining  the precise constant dependence on the $ A_p $ norm of w of the operator norms of singular integrals,  maximal functions,  and other operators in $L^p(w)$ was Buckley in \cite{B}. There he proves that
$$||M||_{L^p(w)}\leq C(n,p ) ||w||_{A_p}^{\frac 1{p-1}}.$$
Recently  T.  Hyt\"{o}nen, C. P\'erez and E. Rela in \cite{HPR}  improved this result  by giving  a sharp weighted bound for
the Hardy-Littlewood maximal  operator  involving  the Fujii-Wilson $A_\infty $ - constant.
 Also  T. Hyt\"onen and C. P\'{e}rez  in \cite{HP}  improved    (\ref{debilTA1}) ,  (\ref{fuerteTA1})  and several well known results,  using for all these cases the Fujii-Wilson $A_\infty $- constant.

In this paper we obtain  similar results as the ones in (\ref{debilTA1}), (\ref{fuerteTA1}) and (\ref{fuerteTAp})   for one-sided weights  and one-sided singular integrals.

A weight  $w\in A_1^+$ if there exists $C>0$ such that $M^-w(x)<C\,w(x)$  $a.e. x\in \mathbb{R}$ (where $M^-$ is the one-sided Hardy-Littlewood maximal operator ), the smallest possible $C$ here is denoted by $||w||_{A_1^+}$. When $n=1$ and $T$ is a Calder\'on-Zygmund singular integral operator with kernel $K$ supported
in $(-\infty,0)$ we say that $T$ is a one-sided singular integral
and we write $T^+$ to emphasize it.  In this paper we obtained the following results

\begin{theor}\label{fuerte}
Let  $1 < p < \infty$, $w\in A^+_1$ and  $T^+$ be a one-side singular integral,
 then,
 \begin{equation}
||T^+f||_{L^p(w)} \leq  C pp' ||w||_{A_1^+}||f||_{L^p(w)},
 \end{equation}
where $C$ only depends on $T^+$.
\end{theor}

\begin{theor}\label{debil}
Let   $w\in A^+_1$ and  $T^+$ be a one-side singular integral,
 then,
 \begin{equation}\label{weak}
||T^+f||_{L^{1,\infty}(w)} \leq C||w||_{A_1^+}\log(e+||w||_{A_1^+})||f||_{L^1(w)},
 \end{equation}
where $C$ only depends on $T^+$.
\end{theor}

\begin{corol}\label{coroldebil}
Let  $1 < p < \infty$, $w\in A^+_p$ and  $T^+$ be a one-side singular integral,  then
\begin{equation}
||T^+f||_{L^{p,\infty}(w)} \leq C||w||_{A_p^+}\log(e+||w||_{A_p^+})||f||_{L^p(w)},
 \end{equation}
where $C=C(p,T^+)$.
\end{corol}
By a duality argument, \textbf{Corollary \ref{coroldebil} } implies the following result:
\begin{corol}\label{corolfuerte}
Let  $1 < p < \infty$, $w\in A^-_p$ and  $T^-$ be a one-side singular integral, then for any measurable set $E$
\begin{equation}
||T^-(\sigma \chi_E)||_{L^{p}(w)} \leq C||w||_{A_p^-}^{\frac{1}{p-1}}\log(e+||w||_{A_p^-})\sigma (E)^{\frac{1}{p}},
 \end{equation}
where $C=C(p,T^-)$ and $\sigma=w^{\frac{-1}{p-1}}$.
\end{corol}

Clearly, every theorem has a corresponding one, reversing the
orientation of $\Bbb R$.

\textbf{Theorems \ref{fuerte}}, \textbf{\ref{debil}} and \textbf{Corollaries \ref{coroldebil}}, \textbf{\ref{corolfuerte}}, for one-sided singular integrals, improve the ones obtained in \cite{ASC}  by putting in the inequalities a wider class of weights (the Sawyer classes) .

The article is organized as follows: in Section 2 we introduce notation, definitions and well known results. In Section 3 we prove some previous lemmas that will be essential to obtain the proofs of Theorems and Corollaries given in Section 4. In Section 5 we give a weaker version and a simplest proof of \textbf{Lemmas \ref{RHI}} and \textbf{\ref{epsilon}} of Section 3.

\section{Preliminaries}

In this section we give some definitions and well known results.
\subsection{One-side singular integral operators and Sawyer's weights}

\begin{defi}
  Let $f\in L_{loc}^1(\re^n)$. The  one-side maximal operators are defined as
\begin{align*}
M^+f(x)=\sup_{h>0}\frac{1}{h}\int_x^{x+h}|f(t)|\,dt , \quad M^-f(x)=\sup_{h>0}\frac{1}{h}\int_{x-h}^{x}|f(t)|\,dt .
\end{align*}
\end{defi}

The good weights for these operators are those of the Sawyer's classes.
We recall this definition.
\begin{defi}
 Let $w$ be a non-negative
locally integrable function and $1\le p<\infty$. We say that $w\in
A_p^+$ if there exists $C_p<\infty$ such that for every $a<x<b$
\begin{equation}
\frac{1}{(b-a)^p}\left(\int_a^xw\right)\left(\int_x^b w^{\frac{-1}{p-1}}\right)^{p-1}\le C_p,
\end{equation}
when $1<p<\infty$, and for $p=1$,
\begin{equation}
M^- w(x)\le C_1\,w(x),
\qquad \mbox{for a.e. }x\in (a,b),
\end{equation}
finally we set $A_\infty^+=\cup_{p\ge 1}A_p^+$.
\end{defi}
The smallest possible $C_1$ in (2.2) here is denoted by $||w||_{A_1^+}$ and the smallest possible $C_p$ in (2.1) here is denoted by $||w||_{A_p^+}$.

 It is well known that the Sawyer classes characterize the boundedness of
the one-sided maximal function on weighted Lebesgue spaces.
Namely, $w\in A_p^+$, $1<p<\infty$, if and only if $M^+$ is bounded on
$L^p(w)$; and $w\in A_1^+$ if and only if $M^+$ maps $L^1(w)$ into
$L^{1,\infty}(w)$. See \cite{S},\cite{R},\cite{ROT}. The classes $A_p^-$ for $1\le p<\infty$  are defined analogously.

We also define
\begin{align*}
M_r^+f(x)=\sup_{h>0}\left(\frac{1}{h}\int_x^{x+h}|f(t)|^r\,dt\right)^\frac{1}{r} , \quad M_r^-f(x)=\sup_{h>0}\left(\frac{1}{h}\int_{x-h}^{x}|f(t)|^r\,dt\right)^\frac{1}{r} ,
\end{align*}
where $r\geq 1$. Observe that $M^+f\leq M_r^+f$ for all $r\geq 1$. Also, we will consider the following maximal operators introduced in \cite{ROT},
\begin{align*}
M_g^+f(x)=\sup_{h>0}\int_x^{x+h}|f(t)|g(t)\,dt\left(\int_x^{x+h}g(t)\,dt\right)^{-1} , \\ M_g^-f(x)=\sup_{h>0}\int_{x-h}^{x}|f(t)|g(t)\,dt\left(\int_{x-h}^{x}g(t)\,dt\right)^{-1} ,
\end{align*}
where $g$ is a positive locally integrable function on $\re$.

The classes $A_p^+(g)$, $1\le p\le \infty$ are defined as  following, let $w$ be non-negative
locally integrable functions and let $1\le p<\infty$. We say that $w\in A_p^+(g)$ if there exists $C_p<\infty$ such that for every $a<x<b$
\begin{equation}\label{Apg}
\left(\int_a^xw\right)\left(\int_x^b g^{p'}\sigma\right)^{p-1}\le C_p\left(\int_a^b g\right)^p,
\end{equation}
where $\sigma=w^{\frac{-1}{p-1}}$, $\frac{1}{p}+\frac{1}{p'}=1$, when $1<p<\infty$, and for $p=1$,
\begin{align*}
M_g^-(g^{-1} w)(x)\le C_1\,g^{-1}w(x),
\qquad \mbox{ a.e. }x\in (a,b).
\end{align*}

In \cite{ROT} it was proved that $w\in A_p^+(g)$, if, and only if $M_g^+$ is bounded from
$L^p(w)$ into $L^p(w)$, for $1<p<\infty$,  and $w\in A_1^+(g)$, if, and only if $M_g^+$ maps $L^1(w)$ into
$L^{1,\infty}(w)$.  Observe that if $g\equiv 1$ then $A_p^+(g)=A_p^+$, for $1\leq p\leq \infty$.

\begin{defi}
We shall say that a function $K$ in  $L^1_{\text{loc}}(\mathbb{R}^n\setminus
\{0\})$ is a Calder\'on-Zygmund kernel if the following properties are satisfied:
\begin{itemize}
\item $||\widehat{K}||_{\infty}<c_1$
\item $|K(x)|<\frac{c_2}{|x|^n}$
\item $|K(x)-K(x-y)|<\frac{c_3|y|}{|x|^{n+1}}$, \quad \mbox{where} $|y|<\frac{|x|}{2}$.
\end{itemize}

The Calder\'{o}n-Zygmund singular integral operator associated to $K$ is defined
\begin{align*}
T(f)=Vp(K \ast f)(x)=\lim_{\epsilon\rightarrow 0} \int_{\re^n\slash B_{\epsilon}(0)}K(x-y)f(y)\,dy
\end{align*}$1<p<\infty$
and the maximal operator  associated with this kernel $K$  is
\begin{align*}
T^*(f)=\sup_{\epsilon> 0} \int_{\re^n\slash B_{\epsilon}(0)}|K(x-y)||f(y)|\,dy
\end{align*}
\end{defi}

A one-sided singular integral $T^+$ is a singular integral associated
to a Calder\'on--Zygmund kernel with support in $(-\infty ,0)$;
therefore, in that case,
$$
T^+f(x)=\lim_{\epsilon\to 0^+}\int_{x+\epsilon}^\infty K(x-y)f(y)\,dy.
$$
Examples of such  kernels are given in \cite{AFR}. In an analogous way we defined $T^-$.

\begin{remark}\label{observacion1}
H. Aimar, L. Forzani and F.J. Mart\'in-Reyes proved in \cite{AFR} that the one-sided singular integral $T^+$ is controlled by the one-side maximal functions $M^+$ in the $L^p(w)$ norm if  $w\in A_\infty^+$.
\end{remark}
\begin{remark}\label{observacion2}
It is  well known to that the classes $ A_p$ are included in $ A_p^+$ and $ A_p^-$; namely  $ A_p= A_p^-\cap A_p^+$.
\end{remark}
\begin{remark}\label{observacion3}
  The one-sided classes of weights satisfy the following  factorization,  $w\in A_p^{+}$ if only if $w=w_1w_2^{1-p}$ with $w_1 \in A_1^+$ and $w_2 \in A_1^-$, and   $||w||_{A_p^+}\leq ||w_1||_{A_1^+} ||w_2||_{A_1^-} ^{p-1}$.
\end{remark}
\begin{remark} \label{observacion4}
  It is easy to check that    $(M^-f)^{\delta}\in A_1^+$ for all $0<\delta<1$ with\\ $||(M^-f)^{\delta}||_{A_1^+}\leq\frac{C}{1-\delta}$.
\end{remark}

\section{Previous Lemmas }\label{section:main}
To obtain \textbf{Theorems  \ref{fuerte}} and  \textbf{\ref{debil}} we need to prove a sharp weak reverse H\"older's inequality  for one-sided weights and also a particular  case of the Coifman-type estimate  for one-sided singular integrals  and one-sided maximal operator.

\subsection{Sharp weak reverse H\"older's inequality}

F.J. Mart\'in-Reyes proved in \cite{R} ( see Lemma 5) a weak reverse H\"older's inequality. Here  we will be more precise in the constants. To make this work self contained  we include the proof of  this Lemma.

\begin{lemma}\label{RHI} {\rm Sharp weak reverse H\"older's inequality.}\\
Let $1\leq p < \infty$,  $w\in A_{p}^+$ and   $r_w=1+\frac{1}{4^{p+2}e^{\frac{1}{e}}||w||_{A_p^+}}$,  when $p>1$, and
$r_w=1+\frac{1}{16e^{\frac{1}{e}}||w||_{A_1^+}}$ when $p=1$,  then

 \begin{equation}
\int_a^b w^{r_w} \leq 2 M^-(w\chi_{(a,b)})(b)^{r_w-1}\int_a^bw,
 \end{equation}
for every bounded interval $(a,b)$, and therefore
\begin{equation}
M^-_{r_w}(w\chi_{(a,b)})(b) \leq 2 M^-(w\chi_{(a,b)})(b).
 \end{equation}
\end{lemma}

\begin{proof} Let $(a,b)$ be a bounded interval, $1\leq p<\infty$ and $w\in A_p^+$. First we will  prove the following statement: let $\lambda>\lambda _0=M^-(w\chi_{(a,b)})(b)$ then

\begin{equation}\label{statement}
w\{x\in(a,b):w(x)>\lambda\}\leq 2\lambda\left|\{x\in(a,b):w(x)>\beta\lambda\}\right|,
 \end{equation}
 where  $\beta=(4^p ||w||_{A_p^+})^{-1}$, for $1<p<\infty$ and $\beta=(||w||_{A_1^+})^{-1}$, for $p=1$.

Let $$\Omega_\lambda=\{x\in (a,b): M^-(w\chi_{(a,b)})(x)>\lambda\}$$ and let ${I_j}=(a_j,b_j)$ be the maximal disjoint intervals whose union give $\Omega_\lambda$. Then $\lambda=\frac{1}{|I_j|}\int_{I_j}w \leq\frac{1}{x-a_j}\int_{a_j}^x w $ for all $x\in(a_j,b_j)$. ( See Lemma 2.1 in \cite{S}).

 Let $1<p<\infty$. In \cite{R} (see Lemma 4) the author prove that if $\lambda>0$ then there exist  $\alpha>0$ and $\beta>0$ such that
\begin{equation}\label{lemma4}
\left|\{x\in(a,b):w(x)>\beta\lambda\}\right|>\alpha(b-a).
 \end{equation}
Following the same steps as in that proof, we observe that if we choose $\beta=(4^p ||w||_{A_p^+})^{-1}$ then $\alpha=\frac 12$. Then using (\ref{lemma4}) for this $\beta$ we get

\begin{align*}
w(\{x\in(a,b):w(x)>\lambda\})&\leq w(\Omega_\lambda)
= \sum_j \int_{I_j}w=\lambda\sum_j (b_j-a_j)\\
&\leq2\lambda\sum_j \left|\{x\in(a_j,b_j):w(x)>\beta\lambda\}\right|\\
&\leq 2\lambda \left|\{x\in(a,b):w(x)>\beta\lambda\}\right|.
\end{align*}

Now for $p=1$,
\begin{align*}
w(\{x\in(a,b):w(x)>\lambda\})&\leq w(\Omega_\lambda)=\sum_j\int_{I_j}w=\lambda\sum_j|I_j|\\
&=\lambda\left|\Omega_\lambda\right| \leq \lambda\left|\{x\in(a,b):w(x)>\beta\lambda\}\right|,
\end{align*}
where the last inequality holds  with $\beta=(||w||_{A_1^+})^{-1}$ and the fact that $w\in A_1^+$.

Now multiplying the inequality (\ref{statement}) by $\lambda^{\delta-1}$ and integrating for $\lambda\geq\lambda_0$ we get,
\begin{align*}
\int_{\lambda_0}^{\infty}\lambda^{\delta-1}\int_{\{x\in(a,b):w(x)>\lambda\}}w(x)\,dx\,d\lambda&\leq2\int_{0}^{\infty}\lambda^{\delta}\left|\{x\in(a,b):w(x)>\beta\lambda\}\right|\,d\lambda\\
&=2\int_{0}^{\infty}\left(\frac{t}{\beta}\right)^{\delta}\frac{1}{\beta}\left|\{x\in(a,b):w(x)>t\}\right|\,dt\\
&=\frac{2}{\beta^{\delta+1}}\int_a^b\int_0^{w(x)}t^{\delta}\,dt\,dx\\
&=\frac{2}{(1+\delta)\beta^{\delta+1}}\int_a^b w(x)^{1+\delta}\,dx
\end{align*}
 with $\delta$  to determinate. On the other hand,
\begin{align*}
\int_{\lambda_0}^{\infty}\lambda^{\delta-1}\int_{\{x\in(a,b):w(x)>\lambda\}}w(x)\,dx\,d\lambda&=\int_{\{x\in(a,b):w(x)>\lambda_0\}}w(x)\int_{\lambda_0}^{w(x)}\lambda^{\delta-1}\,d\lambda\,dx\\
&=\int_{\{x\in(a,b):w(x)>\lambda_0\}}w(x)\frac{1}{\delta}\left(w(x)^{\delta}-\lambda_0^{\delta}\right)\\
&\geq \frac{1}{\delta} \int_a^b w(x)^{1+\delta}\,dx-\frac{\lambda_0^{\delta}}{\delta}\int_a^b w(x)\,dx,
\end{align*}
therefore
\begin{align*}
\left(1-\frac{2\delta}{(1+\delta)\beta^{\delta+1}}\right)\int_a^b w(x)^{1+\delta}\,dx\leq \lambda_0^{\delta} \int_a^b w(x)\,dx.
\end{align*}
If $p>1$ we choose  $\delta=(4^{p+2}e^{\frac{1}{e}}||w||_{A_p^+})^{-1}$ and  if  $p=1$ we choose  $\delta=(16e^{\frac{1}{e}}||w||_{A_1^+})^{-1}$. Then using  that $t^\frac{1}{t}\leq e^\frac{1}{e}$ for all $t\geq 1$,  we get
\begin{align*}
1-\frac{2\delta}{(1+\delta)\beta^{\delta+1}}\geq\frac{1}{2}
\end{align*}
obtaining the desired result with $r_w=1+\delta$.
\end{proof}

\begin{lemma}\label{epsilon}
Let $1<p<\infty$, $w\in A_p^-$, and $a<b<c$ such that $E\subseteq (b,c)$, for $E$ a  measurable set. For all $\epsilon > 0$, there exists $C=C(\epsilon,p)$ such that if $\left|E\right|<e^{-C||w||_{A_p^-}}(b-a)$ then $w(E)<\epsilon w(a,c)$.
\end{lemma}

\begin{proof}
Let $w\in A_p^-$  we apply the analogous   to  \textbf{Lemma \ref{RHI}}, i.e.
\begin{align*}
\int_b^c w^r \leq 2 M^+(w\chi_{(b,c)})(b)^{r-1}\int_b^cw,
\end{align*}
this implies
\begin{align*}
(M_w^+(w^{r-1}\chi_{(b,c)})(b))^{\frac{1}{r-1}}\leq 2 M^+(w\chi_{(b,c)})(b),
\end{align*}
where we take  $r=r_w=1+\frac{1}{4^{p+2}e^{\frac{1}{e}}||w||_{A_p^-}}$,  when $p>1$ and $r=r_w=1+\frac{1}{16e^{\frac{1}{e}}||w||_{A_1^-}}$,
when $p=1$.\\

Using the definition of $M_g^+$, with $g=w$ we have that for all $x\in (a,b)$

\begin{align*}
\left (\frac{1}{w(a,c)}\int_b^c w^{r-1}w\right)^{\frac{1}{r-1}} &\leq (M_w^+(w^{r-1}\chi_{(x,c)})(x))^{\frac{1}{r-1}}\\
&\leq2 M^+(w\chi_{(x,c)})(x)\\
&\leq2 M^+(w\chi_{(a,c)})(x),
\end{align*}
then  $$(a,b)\subseteq\{x:M^+(w\chi_{(a,c)})(x)>\frac{1}{4}\left(\frac{1}{w(a,c)}\int_b^c w^{r}\right)^{\frac{1}{r-1}}\}.$$
Recalling that $M^+$ is of weak type $(1,1)$ with respect to the Lebesgue measure, we get

\begin{align*}
b-a < 12 w(a,c)^{\frac{1}{r-1}}\left(\int_b^c w^{r}\right)^{\frac{-1}{r-1}}w(a,c).
\end{align*}
This last inequality says that $1\in A_{r'}^+(w)$ (where $r'$ is such that $\frac 1r+\frac 1{r'}=1$), with constant $12$, see (\ref{Apg}).
Let $x\in (a,b)$, by hypothesis $E\subset (b,c)$ then

\begin{align*}
M_w^+ (\chi_{E}(x))&\geq \frac{1}{w(x,c)}\int_x^c \chi_{E}(t)w(t)\,dt \geq \frac{w(E)}{w(a,c)},
\end{align*}
so the interval $(a,b)\subset\{x: M_w^+ (\chi_{E}(x))>\frac{w(E)}{2w(a,c)}\}$. Observe that $M_w^+$ is of weak type $(r',r')$ with respect to the Lebesgue measure with constant $K$, (see \cite{ROT}), then
\begin{align*}
  b-a\leq\left|\left\{x: M_w^+ (\chi_{E}(x))>\frac{w(E)}{2w(a,c)}\right\}\right|\leq K \left(\frac{2w(a,c)}{w(E)}\right)^{r'}|E|.
\end{align*}

 Taking into account that $1<r<2$ and $K^{\frac{1}{r'}}\leq \xi$, where $\xi$ not depend on  $p$ nor $||w||_{A_p^-}$ (see \cite{ROT}), then
\begin{equation}\label{epsilon-L3.2}
\frac{w(E)}{w(a,c)}<\left(K\frac{\left|E\right|}{b-a}\right)^{\frac{1}{r'}}<\xi e^{\frac{-C||w||_{A_p^-}}{r'}}<\epsilon,
\end{equation}
where the last inequality holds by choosing an appropriate $C$ depending  only on $ p $.
\end{proof}

\subsection{The Coifman-type estimate }

Now we give a particular case of the Coifman-type estimate. In order to do this  we need a kind of $good-\lambda$ inequality result.

\begin{lemma}\label{buenoslambda}
Let $1 \leq p < \infty$,  $w\in A^-_p$, $T^-$ be a one-side singular integral and  $T^*$ the maximal operator related to  $T^-$.   Then there exist positive constants  $c_1, c_2,$  $\gamma_0 >0$ such that  for every  $0<\gamma<\gamma_0$
\begin{align*}
\left|\{x\in\re: T^*f(x)>2\lambda, M^-f(x)<\gamma\lambda\}\right|<c_1e^{-\frac{c_2}{\gamma}}\left|\{T^*f(x)>\lambda\}\right|
\end{align*}
holds for $f\in L^1(\mathbb{R})$ , $\lambda>0$.  Also,  for  all $\epsilon>0$ , there exists  $c'$ depending  on  $\epsilon,\gamma_0$ and $p$ such that
\begin{align*}
w\left(\left\{x\in\re: T^*f(x)>2\lambda, M^-f(x)<\frac{c'\lambda}{||w||_{A_p^-}}\right\}\right)<\epsilon w(\{T^*f(x)>\lambda\}).
\end{align*}
\end{lemma}

\begin{proof}
Since the set $\{x: T^*f(x)>\lambda\} $ is open and has finite measure, for $f\in L^1(\re)$ , then  it can be written as a disjoint countable union of open intervals. Let $J=(a,b)$ be such an interval. It is enough to prove that there exist $c_1$, $c_2$, $c'$ and $\gamma_0$ such that
\begin{align*}
\left|\{x\in J: T^*f(x)>2\lambda, M^-f(x)<\gamma\lambda\}\right|<c_1e^{-\frac{c_2}{\gamma}}\left|J\right|,
\end{align*}
and
\begin{align*}
w\left(\left\{x\in J: T^*f(x)>2\lambda, M^-f(x)<\frac{c'\lambda}{||w||_{A_p^-}}\right\}\right)<\epsilon w(J),
\end{align*}
for every $0<\gamma<\gamma_0$ and every $\lambda>0$. Let us now take a sequence $\{x_i\}_{i=0}^{\infty}$ in $J=(a,b)$ in such a way that  $x_0=b$ and $x_{i-1}-x_{i}=x_{i}-a$ for every $i>0$. Observe that we only need to prove that
\begin{equation}\label{3.6}
\left|\{x\in (x_{i+1},x_{i}): T^*f(x)>2\lambda, M^-f(x)<\gamma\lambda\}\right|<c_1e^{-\frac{c_2}{\gamma}}(x_{i+1}-x_{i+2}).
\end{equation}

Then using \textbf{Lemma }\ref{epsilon} there exists $c'$ depending on $\epsilon,\gamma_0,p,c_1,c_2$, such that
\begin{equation}
w\left(\left\{x\in (x_{i+1},x_{i}): T^*f(x)>2\lambda, M^-f(x)<\frac{c'\lambda}{||w||_{A_p^-}} \right\}\right) <\epsilon w(x_{i},x_{i+2}).
\end{equation}

 Let us show (\ref{3.6}). Let $i\in \mathbb{N}$, if $\{x\in (x_{i+1},x_{i}): T^*f(x)>2\lambda, M^-f(x)<\gamma\lambda\}=\emptyset$ there is nothing to prove. We choose $\overline{a}<a$ such that $x_i-a=a-\overline{a}$ and
 \begin{align*}
 \xi=\sup\{x\in(x_{i+1},x_i): M^-f(x)\leq\gamma\lambda\}.
 \end{align*}

  Let us write $f=f_1+f_2$ with $f_1=f\chi{(\overline{a},\xi)}$ then
\begin{align*}
&\quad \quad\quad \quad\{x\in (x_{i+1},x_{i}): T^*f(x)>2\lambda, M^-f(x)<\gamma\lambda\}\subset\\
&\{\!x\!\in \!(x_{i+1},\xi)\!: \!\!T^*f_1(x)\!>\!\frac{1}{2}\lambda\!, M^-f(x)\!<\!\gamma\lambda\!\}\!\cup\!\left\{\!x\!\in \!(x_{i+1},\xi)\!:\! T^*f_2(x)\!>\!\frac{3}{2}\lambda, M^-f(x)\!<\!\gamma\lambda\!\right\}.
\end{align*}

The first set on the right term of the formula is essentially empty for $\gamma$ small enough. By standard estimation (see \cite{AFR}), we get that for $x\in(x_{i+1},\xi)$, $T^*f_2(x)\leq\frac{3}{2}\lambda$ then
 $$\left\{x\in (x_{i+1},\xi): T^*f_2(x)>\frac{3}{2}\lambda, M^-f(x)<\gamma\lambda\right\}=\emptyset,$$
 for $0<\gamma<\gamma_0$ small enough.

Now we work with $f_1$. Let $\Omega=\{x\in (x_{i+1},\xi): M^-f_1(x)>3\gamma\lambda\}$, observe that
 \begin{align*}
\int_{\re} f_1(t)\,dt\leq 4\gamma\lambda(x_i-x_{i+1}).
\end{align*}
The last inequality implies that $\Omega\subset(\overline{a},\widetilde{a})$ with $\widetilde{a}-\xi=\frac{4}{3}(x_i-x_{i+1})$. Let us write $\Omega=\bigcup I_j $ where $I_j=(a_j,b_j)$ are disjoint maximal intervals. Then
 \begin{align*}
\frac{1}{\left|I_j\right|}\int_{I_j} f_1(t)\,dt=3\gamma\lambda.
\end{align*}
We define ${I}^+_j=(b_j,c_j)$, $\left|{I}^+_j\right|=2\left|I_j\right|$, $\widetilde{\Omega}=\bigcup ({I}^+_j\cup I_j)= \bigcup \widetilde{I}_j$ and $f_1=g+h$ with
\begin{align*}
g=f_1\chi_{\re/\Omega}+\sum_j 3\gamma\lambda\chi_{I_j}, \qquad h=\sum_j h_j=\sum_j (f_1-3\gamma\lambda)\chi_{I_j}.
\end{align*}
 Observe that $g\leq3\gamma\lambda$ and is supported in $(\overline{a},\widetilde{a})$. Then using \textbf{Lemma 2.11} in \cite{B} i.e.\\
 \\
\emph{Suppose $g\in L^{\infty}(I)$ and that $T$ is an operator for which
 $$|\{x:T\phi(x)>\alpha\}|\leq \left(\frac{C_p||\phi||_p}{\alpha} \right)$$
for all $\phi\in L^p(\re)$ and sufficiently large $p$ and $\alpha$, $C$ being a constant independent of $p$.\\
Then
$$|\{x:Tg(x)>\alpha\}|\leq C e^{\frac{\alpha}{e||g||_{\infty}}}|I|,$$}
 \\
 we have
\begin{align*}
\left|\left\{x:T^*g(x)>\frac{\lambda}{4}\right\}\right|\leq e^{\frac{-c}{\gamma}}(\widetilde{a}-\overline{a})\leq \frac{32}{3} e^{\frac{-c}{\gamma}}(x_{i+1}-x_{i+2}).
\end{align*}

Now let us study $T^*h$ for $x\notin\widetilde{\Omega}$,
\begin{align*}
\left|T^*h(x)\right|&\leq\sum_j \int_{I_j}\left|h_j (y)(K(x-y)-K(x-b_j))\right|\,dy\\
&\leq C_{T^-}\sum_j \int_{I_j}\left|h_j (y)\right|\frac{y-b_j}{(x-b_j)^2}\,dy\\
&\leq\frac{3}{2}C_{T^-}\sum_j\frac{\delta_j}{\delta_j^2+(x-b_j)^2} \int_{I_j}\left|h_j (y)\right|\,dy\\
&\leq9 C_{T^-}\gamma\lambda\sum_j\frac{\delta_j}{\delta_j^2+(x-b_j)^2}\left|I_j\right|\\
&\leq C \gamma\lambda\sum_j\frac{\delta_j^2}{\delta_j^2+(x-b_j)^2},
\end{align*}
where $\delta_j=c_j-a_j$. We write $\Delta(x)=\sum_j\frac{\delta_j^2}{\delta_j^2+(x-b_j)^2}$.

Observe that if $x\in \widetilde{\Omega}$ then $M^-f(x)\geq\gamma\lambda$. In fact, if $x\in I_j$, for some $j$, then by definition of $\Omega$ we have that $3\gamma\lambda<M^-f_1(x)<M^-f(x)$. If $x\in I^+_j$ then
\begin{align*}
3\gamma\lambda&=\frac{1}{\left|I_j\right|}\int_{I_j} f_1(t)\,dt=\frac{x-a_j}{(x-a_j)|I_j|}\int_{a_j}^{x}f(t)\,dt\leq 3M^-f(x).
\end{align*}
By the exponential  Carleson's  estimation,  (see \cite{C}), we have
\begin{align*}
\left|\left\{x\in (x_{i+1},\xi): \Delta(x)>\frac{c}{\gamma}\right\}\right|<Ce^{\frac{-c}{\gamma}}\left|(x_{i+1},\xi)\right|\leq2 Ce^{\frac{-c}{\gamma}}(x_{i+1}-x_{i+2}),
\end{align*}
therefore
\begin{align*}
\left|\left\{x\in (x_{i+1},\xi): T^*h(x)>\frac{1}{4}\lambda, M^-f(x)<\gamma\lambda\right\}\right|\leq Ce^{\frac{-c}{\gamma}}(x_{i+1}-x_{i+2}).
\end{align*}
This,  together  with  our  estimates  for  $f_2$  and  $g$,  is  easily  seen  to  imply  the desired  result.
\end{proof}

\begin{lemma}\label{none}
Let $p\geq1$,  $w\in A_p^-$ and let $T^-$ be a one-side singular integral. Then there exist a constant $C$ which depends on $p$ and $T^-$, such that
\begin{align*}
||T^-f||_{L^{1}(w)} &\leq c||w||_{A_p^-}||M^-f||_{L^{1}(w)}.
\end{align*}
\end{lemma}
\begin{proof}
By \textbf{Lemma \ref{buenoslambda}}, for $\epsilon=\frac{1}{4}$ exists $c'$ such that
\begin{align*}
w\left(\left\{x\in\re: T^*f(x)>2\lambda, M^-f(x)<\frac{c'\lambda}{||w||_{A_p^-}}\right\}\right)<\frac14 w(\{T^*f(x)>\lambda\}).
\end{align*}

Observe that
\begin{align*}
&\int_0^N w(\{T^-f>\lambda\})\,d\lambda \leq2\int_0^{\frac{N}{2}} w(\{T^*f>2\lambda\})\,d\lambda\\
&\leq 2\int_0^{\frac{N}{2}} w\left(\left\{T^*f>2\lambda, M^-f<\frac{c'\lambda}{||w||_{A_p^-}} \right\}\right)d\lambda + 2\int_0^{\frac{N}{2}} w\left(\left\{M^-f\geq\frac{c'\lambda}{||w||_{A_p^-}} \right\}\right)d\lambda\\
&=B_1+B_2.
\end{align*}

For $B_1$, we obtain
\begin{align*}
B_1=2\int_0^{\frac{N}{2}} w\left(\left\{T^*f>2\lambda, M^-f<\frac{c'\lambda}{||w||_{A_p^-}} \right\}\right)\,d\lambda \leq\frac{1}{2}\int_0^{N} w(\{T^*f>\lambda \})\,d\lambda.
\end{align*}
It is easy to see that
\begin{align*}
\frac{1}{2}\int_0^{N} w(\{T^*f>\lambda \})\,d\lambda\leq \frac{2||w||_{A_p^-}}{c'}\int_0^{\frac{Nc'}{2||w||_{A_p^-}}} w(M^-f\geq \lambda \})\,d\lambda,
\end{align*}
then
\begin{align*}
||T^-f||_{L^1(w)}\leq \frac{4||w||_{A_p^-}}{c'} ||M^-f||_{L^1(w)},
\end{align*}
obtaining the desired result.
\end{proof}

\begin{lemma}\label{coif}
Let $T^-$ be a one-side singular integral, $p,r \geq1$. Then  there exist $C$  depending only  on $T^-$ such that
 \begin{equation}
\left|\left|\frac{T^-f}{M^-_rw}\right|\right|_{L^{p'}(M^-_rw)} \leq C p'\left|\left|\frac{M^-f}{M^-_rw}\right|\right|_{L^{p'}(M^-_r w)}.
 \end{equation}
\end{lemma}

It is known that the weigth $(M^-_r w)^{1-p'}$belongs to the $A_{\infty}^-$ class with the corresponding constants independent of $w$. Hence \textbf{Lemma \ref{coif}} is a particular case of the Coifman-type estimate.

\begin{proof}

By duality we have
\begin{align*}
\left|\left|\frac{T^-f}{M^-_rw}\right|\right|_{L^{p'}(M^-_rw)}=\sup_{||h||_{L^{p}(M^-_rw)}=1}\int_\re \left|T^-f\right| h\,dx.
\end{align*}

By \textbf{Lemma 2.3} in \cite{ASC} \emph{choosing $s=p$ and $v=M^-_rw$ there exists an operator $R$ such that
\begin{itemize}
\item $h\leq R(h)$
\item $||R(h)||_{L^{p}(M^-_rw)} \leq 2 ||h||_{L^{p}(M^-_rw)}$
\item $R(h)(M^-_rw)^{\frac{1}{p}}\in A_1$ \qquad\mbox{hence}\qquad$R(h)(M^-_rw)^{\frac{1}{p}}\in A_1^-$\\ \mbox{with}\qquad $||R(h)(M^-_rw)^{\frac{1}{p}}||_{A_1^-}\leq cp'$.
\end{itemize}}
Using the \textbf{Remarks (\ref{observacion3})} and \textbf{(\ref{observacion4})} we have
\begin{align*}
||R(h)||_{A_3^-}&=||R(h)(M^-_rw)^{\frac{1}{p}}[(M^-_rw)^{\frac{1}{2p}}]^{-2}||_{A_3^+}\\
&\leq ||R(h)(M^-_rw)^{\frac{1}{p}}||_{A_1^-} ||(M^-_rw)^{\frac{1}{2p}}||_{A_1^+}^2\\
&\leq c p' (\frac{c}{1-\frac{1}{2pr}})^2\leq Cp'.
\end{align*}

Finally by \textbf{Lemma \ref{none}}, 
\begin{align*}
\int_\re \left|T^-f\right| h\,dx&\leq \int_\re \left|T^-f\right| R(h)\,dx \leq C ||R(h)||_{A_3^-} \int_\re  M^-(f) R(h)\,dx\\
&\leq Cp' \int_\re  \frac{M^-f}{M_r^-w} R(h) M_r^-w\,dx \leq Cp' \left|\left|\frac{M^-f}{M_r^-w}\right|\right|_{L^{p'}(M_r^-w)}||R(h)||_{L^{p}(M^-_rw)}.
\end{align*}
As $||h||_{L^{p}(M^-_rw)}=1$ we have

\begin{align*}
\left|\left|\frac{T^-f}{M^-_rw}\right|\right|_{L^{p'}(M^-_rw)} \leq Cp' \left|\left|\frac{M^-f}{M_r^-w}\right|\right|_{L^{p'}(M_r^-w)}.
\end{align*}

\end{proof}

\section{Proof of the results}\label{section:main}
\subsection{Proof of the Theorems} In order to proof Theorem \ref{fuerte} we first need to prove the following result:

\begin{theor}\label{central}
Let $1 < p < \infty$, $1<r<2$ , $w\in A^+_1$ and $T^+$ be a one-side singular integral. Then
 \begin{equation}
||T^+f||_{L^p(w)} \leq C pp'(r')^{\frac{1}{p'}}||f||_{L_p(M^-_r w)}\ \
 \end{equation}
 where $C$ only depends on $T^+$.
\end{theor}

\begin{proof}
Observe that $(T^{+})^*=T^-$ is the adjoint operator of $T^+$, with kernel supported in $(0,\infty)$. Also observe that as $(M^{-}_rw)\in A_{1}^+ \subset  A_{p}^+$, then $(M^{-}_rw)^{1-p'}\in A_{p'}^-\subset A_{\infty}^- $. Therefore is equivalent to prove
\begin{align*}
\left|\left|\frac{T^{-}f}{M_r^-w}\right|\right|_{L^{p'}(M_r^-w)} \leq C pp'(r')^{\frac{1}{p'}}\left|\left|\frac{f}{w}\right|\right|_{L^{p'}(w)}.
 \end{align*}
By H\"older's inequality
 \begin{align*}
\frac{1}{b-a}\int_a^b f w^{-\frac{1}{p}}w^{\frac{1}{p}}\leq\left(\frac{1}{b-a}\int_a^b w^{r}\right)^{\frac{1}{pr}}\left(\frac{1}{b-a}\int_a^b \left(f w^{-\frac{1}{p}}\right)^{(pr)'}\right)^{\frac{1}{(pr)'}},
 \end{align*}
and taking supremum we get
\begin{align*}
(M^-f(b))^{p'}\leq(M_r^- w(b))^{p'-1}(M^-_{(pr)'}(fw^{-\frac{1}{p}})(b))^{p'},
\end{align*}
then
\begin{align*}
\left|\left|\frac{M^-f}{M_r^-w}\right|\right|_{L^{p'}(M_r^-w)}\leq \left|\left|M^-_{(pr)'}\left(fw^{-\frac{1}{p}}\right) \right|\right|_{L^{p'}}.
\end{align*}
Now using that $|| M^-_k g ||_{L^s}\leq C \left(\frac{s}{k}\right)'^{\frac{1}{k}}||g||_{L^s}$, for $g=fw^{\frac{1}{p}}$, $k=(pr)'$ and $s=p'$ we get
\begin{align*}
\left|\left|\frac{M^-f}{M_r^-w}\right|\right|_{L^{p'}(M_r^-w)}\leq C\left(\frac{rp-1}{r-1}\right)^{1-\frac{1}{pr}}\left|\left|\frac{f}{w}\right|\right|_{L^{p'}(w)}\leq C p\left(\frac{1}{r-1}\right)^{1-\frac{1}{pr}}\left|\left|\frac{f}{w}\right|\right|_{L^{p'}(w)}.
\end{align*}
Observe that $t^{\frac{1}{t}}\leq2$ for $t\geq1$, then
\begin{align*}
\left(\frac{1}{r-1}\right)^{1-\frac{1}{pr}}\leq (r')^{1-\frac{1}{p+1}+\frac{1}{pr'}}\leq 2 (r')^{\frac{1}{p'}}.
\end{align*}
Finally applying \textbf{Lemma \ref{coif}} we get,
\begin{align*}
\left|\left|\frac{T^{-}f}{M_r^-w}\right|\right|_{L^{p'}(M_r^-w)} \leq Cp'\left|\left|\frac{M^-f}{M_r^-w}\right|\right|_{L^{p'}(M_r^-w)}\leq C pp'(r')^{\frac{1}{p'}}\left|\left|\frac{f}{w}\right|\right|_{L^{p'}(w)}.
\end{align*}
\end{proof}

\begin{proof}[Proof of Theorem \ref{fuerte}] This result  is a consequence of \textbf{Theorem \ref{central}}. Using \textbf{Lemma \ref{RHI}}, we observe that $r_w'\lesssim||w||_{A_1^+}$ and $M^-_{r_w}(w\chi_{(a,x)})(x) \leq 2 M^-(w\chi_{(a,x)})(x)\leq 2 ||w||_{A_1^+} w(x)$ $a.e.\ x$. Then
 \begin{align*}
||T^+f||_{L^p(w)} &\leq\left(\sum_{k\in\mathbb{Z}}\int_k^{k+1}|T^+f|^p(x)w\chi_{(k,k+1)}(x)\,dx\right)^{\frac{1}{p}}\\
&\leq C pp'(r_w')^{\frac{1}{p'}}\left( \sum_{k\in\mathbb{Z}}\int_k^{k+1}|f|^p(x)M^-_{r_w }(w\chi_{(k,k+1)})(x)\,dx\right)^{\frac{1}{p}} \\
&= C pp'(r_w')^{\frac{1}{p'}} \left(\sum_{k\in\mathbb{Z}}\int_k^{k+1}|f|^p(x)M^-_{r_w }(w\chi_{(k,x)})(x)\,dx\right)^{\frac{1}{p}} \\
&\leq C pp'(r_w')^{\frac{1}{p'}} \left(2\sum_{k\in\mathbb{Z}}\int_k^{k+1}|f|^p(x)M^-(w\chi_{(k,x)})(x)\,dx\right)^{\frac{1}{p}} \\
&\leq C pp'(||w||_{A_1^+})^{\frac{1}{p'}} ||w||_{A_1^+}^{\frac{1}{p}} ||f||_{L^p( w)}\\
&\leq C pp' ||w||_{A_1^+} ||f||_{L^p( w)}.
\end{align*}
\end{proof}

\begin{proof}[Proof of Theorem \ref{debil}]
Without loss of generality we assume that $0\le f\in L^\infty_c(\re)$.
 Let
\[
\Omega=\{x\in\re :M^+f(x)>\lambda\}
=\bigcup_jI_j=\bigcup_j(a_j,b_j),
\]
where $I_j=(a_j,b_j)$ are the connected component of $\Omega$
and they satisfy
\[
\frac 1{|I_j|}\int_{I_j}f(y)\,dy=\lambda.
\]
Note that if $x\notin\Omega$, then for all $h \ge 0$
\[
\frac 1{h}\int_{x}^{x+h}f(y)\,dy \leq\lambda.
\]
Therefore $f(x)\leq\lambda$ for  a.e $x\in \re\setminus\Omega$.
Let $ I_j^-=(c_j,a_j)$ with $c_j$ chosen so that $| I_j^-|=2\,|I_j|$ and set
$$
\tilde{\Omega}
=
\bigcup_j( I_j^-\cup I_j)
=
\bigcup_j\tilde {I_j}.
$$
We write  $f=g+h$ where
$$
g=f\bigchi_{\R\setminus
\Omega}+\sum_{j=1}^{\infty}\lambda\bigchi_{I_j},
\qquad
h=\sum_{j=1}^{\infty}h_j=\sum_{j=1}^{\infty}(f-\lambda)\bigchi_{I_j}.
$$
Observe that $0\le g(x)\leq \lambda$ for a.e. $x$ and also that $h_j$ has vanishing integral.
Then
\begin{align*}
w(\{x: \left|T^+f(x)\right|>\lambda\})&\leq w(\widetilde{\Omega})+w\left(\left\{x\in\re \setminus\widetilde{\Omega}:\left| T^+h(x)\right|>\frac{\lambda}{2}\right\}\right)\\
&+w\left(\left\{x\in\re \setminus\widetilde{\Omega}: \left|T^+g(x)\right|>\frac{\lambda}{2}\right\}\right)=I+II+III.
\end{align*}

 We estimate $I$:
\begin{align*}
 I=w(\widetilde{\Omega})&\leq\sum_j(w(I^-_j)+ w(I_j)),
\end{align*}
for each $j$
\begin{align*}
w(I^-_j)&=\frac{w(I^-_j)}{\left|I_j\right|}\left|I_j\right|=\frac{w(I^-_j)}{\left|I_j\right|}\frac{1}{\lambda}\int_{I_j}f(x)\,dx\\
&=\frac{1}{\lambda}\int_{I_j}\frac{1}{\left|I_j\right|}\int_{I^-_j}w(t)\,dt \, f(x)\,dx\leq\frac{3}{\lambda}\int_{I_j}\frac{1}{(x-c_{j})}\int_{c_j}^{x}w(t)\,dt \,f(x)\,dx\\
&\leq\frac{3}{\lambda}\int_{I_j}f(x)M^-w(x)\,dx.
\end{align*}
On the other hand,  $(w,M^-w)\in A_1^+$ then $M^+$ is weak type $(1,1)$ with respect to this pair of weights,  then
\begin{align*}
\sum_jw(I_j)=w(\{x: M^+f(x)>\lambda\})<\frac{4}{\lambda}\int_{\re}f(t)M^-w(t)\,dt,
\end{align*}
therefore
\begin{align*}
 I=w(\widetilde{\Omega})\leq\frac{7}{\lambda}\int_{\re}f(t)M^-w(t)\,dt\leq\frac{7}{\lambda}||w||_{A_1^+}\int_{\re}f(t)w(t)\,dt.
\end{align*}

To estimate $II$, let $r_j=|I_j|=| I_j^-|/2$. Now we use that  $h_j$ is supported in $I_j$,  $\int_{I_j}h_j=0$, and  that $K$ is supported in $(-\infty,0)$:
\begin{align*}
 II&=w\left(\left\{x\in\re\setminus\widetilde{\Omega}:\left| T^+h(x)\right|>\frac{\lambda}{2}\right\}\right)\leq\frac{2}{\lambda}\int_{\re\setminus\widetilde{\Omega}}\left|T^+h(t)\right|w(t)\,dt\\
  &\leq\frac{2}{\lambda}\sum_j\int_{I_j}\left|h_j(y)\right|\int_{\re\setminus\widetilde{I}_j}\left|K(t-y)-K(t-a_j)\right|w(t)\,dt\,dy\\
 &=\frac{2}{\lambda}\sum_j\int_{I_j}\left|h_j(y)\right|\int_{-\infty}^{c_j}\left|K(t-y)-K(t-a_j)\right|w(t)\,dt\,dy.
\end{align*}
Observe that it is suffices to obtain that  for all $y\in I_j$,
\begin{align*}
 \int_{-\infty}^{c_j}\left|K(t-y)-K(t-a_j)\right|w(t)\,dt\leq C \essinf_{I_j} M^-(w\chi_{\re\setminus\widetilde{I}_j}).
\end{align*}
To see this we use the condition of the kernel $K$,
\begin{align*}
\int_{-\infty}^{c_j}\left|K(t-y)-K(t-a_j)\right|&w(t)\,dt=\sum_{k=1}^{\infty}\int_{a_j-2^{k+1}r_j}^{a_j-2^{k}r_j}\left|K(t-y)-K(t-a_j)\right|w(t)\,dt\\
&\leq K_{T^+} \sum_{k=1}^{\infty}\int_{a_j-2^{k+1}r_j}^{a_j-2^{k}r_j}\left|\frac{y-a_j}{(t-a_j)^2}\right|w(t)\,dt\\
&\leq K_{T^+}
\sum_{k=1}^{\infty}\frac{y-a_j}{(2^kr_j)^2}\int_{a_j-2^{k+1}r_j}^{a_j-2^{k}r_j}w(t)\chi_{(a_j-2^{k}r_j,a_j-2^{k+1}r_j)}\,dt\\
&\leq K_{T^+} \sum_{k=1}^{\infty}\frac{1}{2^k}\frac{1}{(2^kr_j)}\int_{a_j-2^{k+1}r_j}^{a_j-2^{k}r_j}w(t)\chi_{(a_j-2^{k}r_j,a_j-2^{k+1}r_j)}\,dt,
\end{align*}
where $K_{T^+}$ depends only on of $T^+$. If $x\in I_j$
\begin{align*}
&\int_{-\infty}^{c_j}\left|K(t-y)-K(t-a_j)\right|w(t)\,dt\\
&\leq K_{T^+}\sum_{k=1}^{\infty}\frac{1}{2^k}\frac{x-a_j+2^{k+1}r_j}{2^kr_j}\frac{1}{x-a_j+2^{k+1}r_j}\int_{a_j-2^{k+1}r_j}^{x}w(t)\chi_{(a_j-2^{k}r_j,a_j-2^{k+1}r_j)}\,dt\\
&\leq K_{T^+} M^-w\chi_{\re\setminus\widetilde{I}_j}(x) \sum_{k=1}^{\infty}\frac{1}{2^k}\left(\frac{x-a_j}{2^kr_j}+\frac{2^{k+1}r_j}{2^kr_j}\right)\leq C M^-(w\chi_{\re\setminus\widetilde{I}_j})(x),
\end{align*}
therefore
\begin{align*}
II&\leq \frac{C}{\lambda}\sum_j  \essinf _{I_j} M^-(w\chi_{\re\setminus\widetilde{I}_j})\int_{I_j}\left|h_j(y)\right|\,dy\\
&\leq \frac{C}{\lambda}\sum_j   \int_{I_j}\left|h_j(y)\right|M^-(w\chi_{\re\setminus\widetilde{I}_j})(y)\,dy\\
&\leq \frac{C}{\lambda}\left[\sum_j   \int_{I_j}f(y)M^-(w\chi_{\re\setminus\widetilde{I}_j})(y)\,dy+\sum_j   \int_{I_j}\left|g(y)\right|M^-(w\chi_{\re\setminus\widetilde{I}_j})(y)\,dy\right]\\
&=\frac{C}{\lambda}(A+B).
\end{align*}

For $ A $ there is nothing to prove.
To work with $ B $ we need to prove the following inequality
\begin{equation}\label{desMrI}
  M^-(w\chi_{\re\setminus\widetilde{I}_j})(y) \leq \frac{3}{2}\essinf_{z\in I_j}\ M^-(w\chi_{\re\setminus\widetilde{I}_j})(z),
\end{equation}
for all $y\in I_j$. In fact for $y, z \in I_j$,
\begin{align*}
M^-(w\chi_{\re\setminus\widetilde{I}_j})(y)&=  \sup_{t<y} \frac{1}{y-t} \int_{t}^y w(s)\chi_{\re\setminus\widetilde{I}_j}(s)\,ds= \sup_{t<c_j} \frac{1}{y-t} \int_{t}^{c_j}  w(s)\chi_{\re\setminus\widetilde{I}_j}(s)\,ds\\
&\leq \sup_{t<c_j} \frac{3}{2}\frac{1}{z-t} \int_{t}^{c_j}  w(s)\chi_{\re\setminus\widetilde{I}_j}(s)\,ds\leq\frac{3}{2}M^-(w\chi_{\re\setminus\widetilde{I}_j})(z).
\end{align*}
Then
\begin{align*}
B&=\sum_j   \int_{I_j}\left|g(y)\right|M^-(w\chi_{\re\setminus\widetilde{I}_j})(y)\,dy=\sum_j   \int_{I_j}\lambda M^-(w\chi_{\re\setminus\widetilde{I}_j})(y)\,dy\\
&\leq \sum_j    \int_{I_j} f(t)\,dt \frac{1}{|I_j|} \int_{I_j}  M^-(w\chi_{\re\setminus\widetilde{I}_j})(y)\,dy\\
&\leq \frac{3}{2}\ \sum_j    \int_{I_j} f(t)\,dt\ \essinf_{I_j} M^-(w\chi_{\re\setminus\widetilde{I}_j})\leq \frac{3}{2}\ \sum_j   \int_{I_j}f(t)M^-(w\chi_{\re\setminus\widetilde{I}_j})(t)\,dt.
\end{align*}

So
\begin{align*}
II\leq \frac{C}{\lambda} \sum_j   \int_{I_j}f(t)M^-(w\chi_{\re\setminus\widetilde{I}_j})(t)\,dt\leq \frac{C}{\lambda}||w||_{A_1^+} \int_{\re}f(t)w(t)\,dt.
\end{align*}

Finally we estimate $III$. First observe that doing the same proof that in (\ref{desMrI}) we obtain
\begin{equation}\label{desMrOmega}
  M^-(w\chi_{\mathbb{R}\setminus\widetilde{\Omega}})(y) \leq \frac{3}{2}\essinf_{z\in I_j}\ M^-(w\chi_{\mathbb{R}\setminus\widetilde{\Omega}})(z),
\end{equation}
for all $y\in I_j$.

 By Chevichef's inequality, using the fact that $g\leq\lambda$ and choosing $r=r_w=1+\frac{1}{16e^{\frac{1}{e}}||w||_{A_1^+}}$  in order to apply \textbf{Theorem \ref{central}} and \textbf{Lemma \ref{RHI}}, we get that
\begin{align*}
III&=w\left(\left\{x\in\re\setminus\widetilde{\Omega}: \left|T^+g\right|(x)>\frac{\lambda}{2}\right\}\right)\leq \frac{2^p}{\lambda^p} \int_{\re\setminus\widetilde{\Omega}} (\left|T^+g\right|(x))^pw(x)\,dx\\
&\leq \frac{2^p}{\lambda^p} \sum_{k\in\mathbb{Z}}\int_{k}^{k+1} (\left|T^+g\right|(x))^pw(x)\chi_{(\re\setminus\widetilde{\Omega})}(x) \chi_{(k,k+1)}(x)\,dx\\
&\leq \frac{2^p}{\lambda^p} (C pp'(r')^{\frac{1}{p'}})^p \sum_{k\in\mathbb{Z}}\int_{k}^{k+1} (\left|g\right|(x))^p M_r^-(w\chi_{(\re\setminus\widetilde{\Omega})} \chi_{(k,k+1)})(x)\,dx\\
&\leq \frac{2^p}{\lambda} (C pp'(r')^{\frac{1}{p'}})^p \sum_{k\in\mathbb{Z}}\int_{k}^{k+1} (\left|g\right|(x)) M_r^-(w\chi_{(\re\setminus\widetilde{\Omega})} \chi_{(k,x)})(x)\,dx\\
&\leq \frac{2^{p+1}}{\lambda} (C pp'(r')^{\frac{1}{p'}})^p \sum_{k\in\mathbb{Z}}\int_{k}^{k+1} (\left|g\right|(x)) M^-(w\chi_{(\re\setminus\widetilde{\Omega})} \chi_{(k,x)})(x)\,dx\\
&\leq \frac{2^{p+1}}{\lambda} (C pp'(r')^{\frac{1}{p'}})^p \int_{\re} (\left|g\right|(x)) M^-(w\chi_{(\re\setminus\widetilde{\Omega})} )(x)\,dx\\
&\leq \frac{2^{p+1}}{\lambda} (C pp'((r')^{\frac{1}{p'}})^p \left[\int_{\re  / \Omega } |g(x)| M^-(w\chi_{\re\setminus\widetilde{\Omega}})(x)\,dx + \int_{\Omega } |g(x)| M^-(w\chi_{\re\setminus\widetilde{\Omega}})(x)\,dx \right].
\end{align*}

Recalling that $f(x)=g(x)$ for all $x\in \re \setminus\Omega$ and arguing as in $B$ for the integral of $g$ in $\Omega$, this time using (\ref{desMrOmega}), we obtain
$$
\int_{\re  } |g(x)| M^-(w\chi_{\re\setminus\widetilde{\Omega}})(x)\,dx\leq \frac 32 \int_{\re } f(x)M^-(w)(x)\,dx.
$$
Now the fact  that $w\in A_1^+$ and $r'=r_w' \leq C ||w||_{A_1^+}$ implies
\begin{align*}
III&\leq \frac{2^{p+1}}{\lambda} (C pp'((r')^{\frac{1}{p'}})^p \int_{\re} f(x) M^-w(x)\,dx\\
&\leq \frac{2^{p+1}}{\lambda} (C pp'(||w||_{A_1^+})^{\frac{1}{p'}})^p  ||w||_{A_1^+} \int_{\re} f(x) w(x)\,dx\\
&\leq \frac{C^p2^{p+1}}{\lambda}  [pp'||w||_{A_1^+}]^p \int_{\re} f(x) w(x)\,dx.
\end{align*}
We take $p=1+\frac{1}{\log(e+||w||_{A_1^+})}$ and observing that $t^{(\log(e+t))^{-1}}$ and $t^{t^{-1}}$ are bounded for $t>1$ we have
\begin{align*}
 [pp'||w||_{A_1^+}]^p\leq C\log(e+||w||_{A_1^+}) ||w||_{A_1^+},
\end{align*}
and as $1<p<2$ we obtain
\begin{align*}
III&\leq \frac{C }{\lambda}\log(e+||w||_{A_1^+})||w||_{A_1^+} \int_{\re} f(x) w(x)\,dx.
\end{align*}
Combining this estimate with $I$ and $II$ completes the proof.\\
\end{proof}

\subsection{Proof of the Corollaries}

We shall need the following Lemma. The equivalent to weights in $A_p$ is proven in \cite{DGPP}.
\begin{lemma}\label{RF}
Let $1<q<\infty$ and let $w\in A_q^+$. Then there exists a nonnegative sublinear operator $D$ bounded on $L^{q'}$ such that for any nonnegative $h\in L^{q'}(w)$:
\begin{enumerate}
\item $h\leq D(h)$;
\item $||D(h)||_{L^{q'}(w)}\leq 2||h||_{L^{q'}(w)}$;
\item $D(h).w\in A_1^+$ \quad with \quad $||D(h).w||_{A_1^+}\leq Cq2^{q}||w||_{A_q^+}$,
\end{enumerate}
where the constant $C$ not depend on $||w||_{A_1^+}$ and $q$.
\end{lemma}

\begin{proof}
In \cite{RT1}  F.J. Mart\'in-Reyes and A. de la Torre   proved that

 \emph{If $w\in A_p^+$ then}
 \begin{equation}\label{BucM+}
 ||M^+||_{L^p(w)}\leq Cp'2^{p'}||w||_{A_p^+}^{\frac{1}{p-1}}.
 \end{equation}

This result is  an analogous one,  for the one-sided maximal operator $M^+$, to the one obtain in  \cite{B},  see \textbf{Theorem 2.5}.  In this work  S. M.  Buckley  gives a sharp estimation in norm $L^p$ of the Hardy-Littlewood maximal operator $M$ respect to weights $w\in A_p$.
We define the operator $S(h)=w^{-1}M^-(|h|w)$, then $S$ is bounded in $L^{q'}(w)$, moreover, $||S||_{L^{q'}(w)}\leq C q2^q ||w||_{A_q^+}$, indeed using the analogous version for $M^-$ of (\ref{BucM+}) we get,
\begin{align*}
||Sh||_{L^{q'}(w)}&=\left(\int_{\re}(w^{-1}M^-(|h|w))^{q'}w\,dx \right)^{\frac{1}{q'}}=\left(\int_{\re}(M^-(|h|w))^{q'}w^{1-q'}\,dx \right)^{\frac{1}{q'}}\\
&\leq ||M^-||_{L^{q'}(w^{1-q'})}|||h|w||_{L^{q'}(w^{1-q'})}\leq Cq2^q||w^{1-q'}||_{A_{q'}^-}^{\frac{1}{q'-1}}||h||_{L^{q'}(w)}.
\end{align*}
Recalling that $w\in A_q^+$ implies $w^{1-q'}\in A_{q'}^-$ and that $||w^{1-q'}||_{A_{q'}^-}=||w||_{A_q^+}^{\frac{1}{q-1}}$, we get
 $||S||_{L^{q'}(w)}\leq C q 2^q||w||_{A_q^+}$, as claimed.

Now we define the operator $D$ via the following convergent Neumann series:
\begin{align*}
D(h)=\sum_{k=0}^{\infty} \frac{S^k(h)}{2^k||S||^k}, \quad \quad \mathrm{where} \quad ||S||=||S||_{L^{q'}(w)}.
\end{align*}
Then $(1)$ and $(2)$ are clearly satisfied.\\ $(3)$ It follows from the definition of $D$ and the sublinearity of $S$ that
\begin{align*}
S(D(h))\leq 2||S||(D(h)-h)\leq 2||S||D(h),
\end{align*}
therefore
\begin{align*}
M^-(D(h)w)=M^-(D(h)w)w^{-1}w=S(D(h))w\leq2||S||D(h)w\leq c q2^q ||w||_{A_q^+} D(h)w.
\end{align*}
\end{proof}

\begin{proof}[Proof of Corollary \ref{coroldebil}]
For $\alpha>0$ we set $\Omega_\alpha=\{x\in\re :|T^+f(x)|>\alpha\}$ and let $\varphi(t)=t\log(e+t)$. Applying \textbf{Lemma \ref{RF}} with $q=p$, we get a sublinear operator $D$ bounded on $L^{p'}$ satisfying properties (1), (2), and (3). Using these properties and {\bf Theorem (\ref{debil})}, we obtain
\begin{align*}
\int_{\Omega_\alpha} h\,w\,dx&\leq \int_{\Omega_\alpha} D(h)\,w\,dx\leq\frac{C}{\alpha}\varphi(||D(h).w||_{A_1^+})||f||_{L^1(D(h).w)}\\
&\leq\frac{C}{\alpha}\varphi(Cp2^p||w||_{A_p^+})\int_{\re}|f|D(h)\,wdx\\
&\leq\frac{C}{\alpha}2\varphi(Cp2^p)\varphi(||w||_{A_p^+})\left(\int_{\re}|f|^p\,wdx\right)^{\frac{1}{p}}\left(\int_{\re}D(h)^{p'}\,wdx\right)^{\frac{1}{p'}}\\
&\leq\frac{C}{\alpha}\varphi(||w||_{A_p^+})||f||_{L^p(w)}||h||_{L^{p'}(w)},
\end{align*}
taking the supremum over all $h$ with $||h||_{L^{p'}(w)}=1$ completes the proof.
\end{proof}

\begin{proof}[Proof of Corollary \ref{corolfuerte}]
Given the one-sided singular operator $T^-$, its adjoint operator is $T^+$. Let $w\in A_p^-$ then $\sigma\in A_{p'}^+$ with $||\sigma||_{A_{p'}^+}=||w||_{A_p^-}^{\frac{1}{p-1}}$.
Applying  \textbf{Corollary \ref{coroldebil}} to the one-sided singular operator $T^+$ and the weight $\sigma$ we get
\begin{align*}
||T^+||_{L^{p',\infty}(\sigma)}&\leq C ||w||_{A_p^-}^{\frac{1}{p-1}}\log\left(e+||w||_{A_p^-}^{\frac{1}{p-1}}\right)||f||_{L^{p'}(\sigma)}\\
&\leq C ||w||_{A_p^-}^{\frac{1}{p-1}}\log(e+||w||_{A_p^-})||f||_{L^{p'}(\sigma)}.
\end{align*}
From this, by duality we obtain
\begin{align*}
||T^-||_{L^{p}(w)}\leq C ||w||_{A_p^-}^{\frac{1}{p-1}}\log(e+||w||_{A_p^-})\left|\left|\frac{f}{\sigma}\right|\right|_{L^{p,1}(\sigma)},
\end{align*}
where $L^{p,1}(\sigma)$ is the standard weighted lorentz space. Setting here $f=\sigma\chi_E$, where $E$ is any measurable set, completes the proof.
\end{proof}

\section{Appendix}\label{section:main}

In this section we will give  another version of  \textbf{Lemma \ref{RHI}}  which will allow us to give an easier proof of a slight weak version   of \textbf{Lemma \ref{epsilon}}.

\begin{lemma}\label{RHI2} {\rm One-sided RHI}. \\
Let  $1\leq p < \infty$, $w\in A_{p}^+$ and $a<b<c$ with $b-a=2(c-b)$. If  $r=1+\frac{1}{4^{p+2}e^{\frac{1}{e}}||w||_{A_p^+}}$ , for  $p>1$ and  $r=1+\frac{1}{16e^{\frac{1}{e}}||w||_{A_1^+}}$, for  $p=1$ then
\begin{align*}
\frac{1}{b-a}\int_a^b w^r \leq \frac{27}{4} \left(\frac{1}{c-a}\int_a^cw\right)^r.
 \end{align*}
\end{lemma}

\begin{proof}
By \textbf{Lemma \ref{RHI}}
 \begin{align*}
\int_{a}^{x} w^r \leq 2 M^-(w\chi_{(a,x)})(x)^{r-1}\int_a^x w.
 \end{align*} for every bounded interval  $(a,x)$. Then
for every  $x\in(b,c)$, we have
 \begin{align*}
\int_a^b w^r \leq \int_a^x w^r \leq 2 M^-(w\chi_{(a,x)})(x)^{r-1}\int_a^xw\leq 2 M(w\chi_{(a,c)})(x)^{r-1}\int_a^cw.
 \end{align*}
Then  $(b,c)\subset\left\{x: M(w\chi_{(a,c)})(x)\geq \left(\frac{\int_a^b w^r}{2\int_a^cw}\right)^{\frac{1}{r-1}}\right\}$. Using that  $M$ is of weak type $(1,1)$ we obtain,
 \begin{align*}
c-b\leq 3 \left(\frac{2\int_a^cw}{\int_a^b w^r}\right)^{\frac{1}{r-1}}\int_a^c w, \quad \quad \text{then}\\
(c-b)\left(\int_a^b w^r\right)^{\frac{1}{r-1}}\leq 2^{\frac{1}{r-1}} 3 \left({\int_a^c w}\right)^{\frac{r}{r-1}}.
 \end{align*}
Observing that $1<r<2$, we get
  \begin{align*}
\frac{1}{b-a}\int_a^b w^r& = \frac{1}{2(c-b)}\int_a^b w^r \leq 2\, 3^{r-1}\frac{1}{2(c-b) (c-b)^{r-1}} \left({\int_a^c w}\right)^{r}\\
&=3^{r-1} \left( \frac{3}{2(c-a)} {\int_a^c w}\right)^{r}=\frac{3^{2r-1}}{2^r} \left( \frac{1}{(c-a)} {\int_a^c w}\right)^{r}\\
&\leq \frac{27}{4} \left( \frac{1}{(c-a)} {\int_a^c w}\right)^{r}.
 \end{align*}
\end{proof}

The following Lemma  is a slight weak version of \textbf{Lemma \ref{epsilon}}.

\begin{lemma}\label{epsilon2}
Let $p\geq1$,  $w\in A_p^-$ ,   $a<b<c$ such that $2(b-a)=(c-b)$ and   $E\subseteq (b,c)$  a measurable set. Then for every  $\epsilon > 0$ there exists $C=C(\epsilon,p)$ such that  if  $\left|E\right|<e^{-C||w||_{A_p^-}}(b-a)$ then  $w(E)<\epsilon w(a,c)$.
\end{lemma}

\begin{proof} We will use the analogous to  \textbf{Lemma \ref{RHI2}} for $A_p^-$ weights.
\begin{align*}
w(E)&=\frac{1}{c-b}\int_b^c w\chi_E \,(c-b)\leq(c-b)\left(\frac{1}{c-b}\int_b^c w^r\right)^{\frac{1}{r}}\, \left(\frac{1}{c-b}\int_b^c \chi_E^{r'}\right)^{\frac{1}{r'}}\\
&=\left(\frac{|E|}{c-b}\right)^{\frac{1}{r'}}(c-b)\frac{27}{4}\frac{1}{c-a}\int_a^c w\leq\left(\frac{|E|}{b-a}\right)^{\frac{1}{r'}}\frac{27}{4}\frac{2}{3 }\int_a^c w \leq\epsilon w(a,c),
 \end{align*}
  where the last inequality is obtained by following the same steps as in (\ref{epsilon-L3.2}).
\end{proof}

As a  Corollary of   \textbf{Lemma \ref{RHI2}}  we obtain another proof of  \textbf{Proposition 3} in \cite{R}, this is

\begin{corol}
Let  $1< p < \infty$  and $w\in A_p^+$ then  $w\in A_{p-\epsilon}^+$, with $p-\epsilon=\frac{p-1}{r(\sigma)}+1$
where $\sigma=w^{1-p'}$ and $r(\sigma)$ is the one obtained  in the analogous version of \textbf{Lemma \ref{RHI2}} for a weight in  $A_{p'}^-$.
\end{corol}

\begin{proof}
In  \cite{ST} it is proved  that   $w\in A_p^+$ if, and only if  there exists $C>0$ such that
 \begin{equation}
 \sup_{a,b,c,d}\frac{1}{(b-a)^p}\left(\int_a^bw\right)\left(\int_c^d w^{\frac{-1}{p-1}}\right)^{p-1}<C.
 \end{equation}
where the supremum   is taken over all  $a,b,c,d$ such that  $a<b<c<d$ and  $2(b-a)=2(d-c)=c-b$.

Let $r=r(\sigma)$ be the one of \textbf{Lemma \ref{RHI2}} and   $a,b,c,d$  as in the previous line, then
\begin{align*}
&\left(\frac{1}{b-a}\int_a^bw \right) \left(\frac{1}{d-c}\int_c^d w^{\frac{-1}{p-\epsilon-1}}\right)^{p-\epsilon-1}\leq \left(\frac{1}{b-a}\int_a^bw \right) \left(\frac{1}{d-c}\int_c^d \sigma^r\right)^{\frac{p-1}{r}}\\
&\leq \left(\frac{1}{b-a}\int_a^bw \right) \left(\frac{1}{d-b}\frac{27}{4}\int_b^d \sigma\right)^{p-1}\leq \frac{1}{4}\left(\frac{81}{64}\right)^{p-1}||w||_{A_p^+}.
 \end{align*}
\end{proof}

\end{document}